\documentclass{amsart}

\usepackage{amsmath, amssymb,epic,graphicx,mathrsfs,enumerate}
\usepackage[all]{xy}

\usepackage{latexsym}
\usepackage{longtable}
\usepackage{epsfig}
\usepackage{hhline}
\usepackage{color}

\def\F{\mathbin{\mathbb{F}}}
\def\N{\mathbin{\mathbb{N}}}
\def\Z{\mathbin{\mathbb{Z}}}
\def\M{\mathrm{M}}
\def\diag{\mathbin{\mathrm{Diag}}}
\def\charp{\mathbin{\mathrm{charpoly}}}
\def\minp{\mathbin{\mathrm{minpoly}}}
\def\GL{\mathrm{GL}}
\def\SL{\mathrm{SL}}

\newtheorem{proposition}{Proposition}[section]
\newtheorem{claim}[proposition]{Claim}
\newtheorem{fact}[proposition]{Fact}

\newtheorem{lemma}[proposition]{Lemma}

\newtheorem{theorem}[proposition]{Theorem}

\begin{document}
\title[An improved diameter bound for finite simple groups of Lie type]{An 
improved diameter bound for finite simple groups of Lie type}

\author[Zolt\'an Halasi]{Zolt\'an Halasi}
\address{Department of Algebra and Number Theory, E\"otv\"os University, 
P\'azm\'any P\'eter S\'et\'any 1/c, H-1117, Budapest, Hungary \and 
Alfr\'ed R\'enyi Institute of Mathematics, Hungarian Academy of Sciences, 
Re\'altanoda utca 13-15, H-1053, Budapest, Hungary}
\email{zhalasi@cs.elte.hu \and halasi.zoltan@renyi.mta.hu}

\author[Attila Mar\'oti]{Attila Mar\'oti}
\address{Alfr\'ed R\'enyi Institute of Mathematics, Hungarian Academy of Sciences, Re\'altanoda utca 13-15, H-1053, Budapest, Hungary}
\email{maroti.attila@renyi.mta.hu}

\author[L\'aszl\'o Pyber]{L\'aszl\'o Pyber}
\address{Alfr\'ed R\'enyi Institute of Mathematics, Hungarian Academy of Sciences, Re\'altanoda utca 13-15, H-1053, Budapest, Hungary}
\email{pyber.laszlo@renyi.mta.hu}

\author[Youming Qiao]{Youming Qiao}
\address{Centre for Quantum Software and Information, Faculty of Engineering and Information Technology,
University of Technology Sydney, Sydney, NSW 2007, Australia} 
\email{Youming.Qiao@uts.edu.au}

\date{November 19, 2018}
\thanks{This work on the project leading to this application has
received funding from the European Research Council (ERC) under the
European Union's Horizon 2020 research and innovation programme
(grant agreement No. 741420). The first, second and third authors were 
partly supported by the National Research, Development and Innovation Office 
(NKFIH) Grant No.~K115799. The first and second authors were also 
supported by the J\'anos Bolyai Research Scholarship of the Hungarian Academy of 
Sciences. The fourth author was also supported by the Australian Research Council 
DE150100720.}

\begin{abstract}
For a finite group $G$, let $\mathrm{diam}(G)$ denote the maximum diameter of a connected Cayley graph of $G$. A well-known conjecture of Babai states that $\mathrm{diam}(G)$ is bounded by ${(\log_{2} |G|)}^{O(1)}$ in case $G$ is a non-abelian finite simple group. Let $G$ be a finite simple group of Lie type of Lie rank $n$ over the field $\F_{q}$. Babai's conjecture has been verified in case $n$ is bounded, but it is wide open in case $n$ is unbounded. Recently, Biswas and Yang proved that $\mathrm{diam}(G)$ is bounded by $q^{O( n {(\log_{2}n + \log_{2}q)}^{3})}$. We show that in fact $\mathrm{diam}(G) < q^{O(n {(\log_{2}n)}^{2})}$ holds. Note that our bound is significantly smaller than the order of $G$ for $n$ large, even if $q$ is large. As an application, we show that more generally $\mathrm{diam}(H) < q^{O( n {(\log_{2}n)}^{2})}$ holds for any subgroup $H$ of $\mathrm{GL}(V)$, where $V$ is a vector space of dimension $n$ defined over the field $\F_q$.      
\end{abstract}
\maketitle

\section{Introduction}

Given a finite group $G$ and a set $S$ of generators of $G$, the associated Cayley graph $\Gamma$ is defined to have vertex set $G$ and edge set $\{ \{g,gs\} : g \in G, \ s \in S \}$. The diameter $\mathrm{diam}_{S}(G)$ of $\Gamma$ is the maximum over $g \in G$ of the length of a 
shortest expression of $g$ as a product of generators in $S$ and their inverses. The maximum of $\mathrm{diam}_{S}(G)$, as $S$ runs over all possible generating sets of $G$, is denoted by $\mathrm{diam}(G)$.    

In 1992 Babai \cite{BS} conjectured that $\mathrm{diam}(G) < {(\log |G|)}^{O(1)}$ holds for any non-abelian finite simple group $G$. (Here and throughout the paper the base of the logarithms will always be $2$, unless otherwise stated.) The first class of simple groups for which Babai's conjecture was proved \cite{Helfgott} were the groups $\mathrm{PSL}(2,p)$ where $p$ is prime. Following Helfgott's paper \cite{Helfgott}, the conjecture was verified for finite simple groups of Lie type of bounded rank by Pyber and Szab\'o \cite{PSz} and Breuillard, Green, Tao \cite{BGT}. In particular, Babai's conjecture holds for exceptional simple groups of Lie type. However the conjecture remains wide open for finite simple groups of Lie type of large rank, that is, for simple classical groups of large rank.

Babai's conjecture is open even in the case of alternating groups. Babai and Seress \cite{BS1} proved that $\mathrm{diam}(\mathrm{A}_{n}) < \exp(\sqrt{n \ln n} (1+o(1)))$ and in \cite{BS} they showed that the same bound holds for arbitrary permutation groups of degree $n$.

The strongest bound to date is $\mathrm{diam}(\mathrm{A}_{n}) < \exp(O {(\log n)}^{4} \log \log n)$ ($n > 2$), due to Helfgott and Seress \cite{HS}. The same estimate is shown to hold in \cite{HS} for arbitrary transitive groups of degree $n$. The inductive proof of Helfgott and Seress relies heavily on the fact that their result extends to transitive groups. For a greatly simplified argument see \cite{He}.  

In connection with Babai's conjecture, we remark that Breuillard and Tointon \cite{BT} showed, without the use of the classification theorem of finite simple groups, that for any $\epsilon > 0$ there is a constant $C_{\epsilon}$ depending only on $\epsilon$ such that every non-abelian finite simple group $G$ with a symmetric generating set $S$ satisfies $$\mathrm{diam}_{S}(G) \leq \max \Big\{ {\Big(\frac{|G|}{|S|}\Big)}^{\epsilon}, C_{\epsilon} \Big\}.$$ Breuillard remarks in his ICM survey \cite{Bre}, that it would be interesting to get non-trivial bounds for all finite simple groups of Lie type also when the rank grows and see if one can improve the above ``crude bound".

Let $G$ be a finite simple group of Lie type of Lie rank $n$ defined over $\F_q$. Biswas and Yang \cite{BY} proved that $\mathrm{diam}(G) < q^{O( n {(\log n + \log q)}^{3})}$. The first result of the present paper provides an improvement of this bound showing that the exponent need not depend on $q$. 

\begin{theorem}
\label{firstmain}
If $G$ is a finite simple group of Lie type of Lie rank $n$ defined over the field of size $q$, then $\mathrm{diam}(G) < q^{O( n {(\log n)}^{2})}$.
\end{theorem}


Let $\Gamma$ be a finitely generated group and $S$ a finite set of generators of $\Gamma$. For a positive integer $n$, let $\gamma_{S}(n)$ denote the number of elements in $\Gamma$ which may be expressed as a product of $n$ elements of $S \cup S^{-1} \cup \{ 1 \}$. The celebrated theorem of Gromov \cite{Gromov} asserts that $\Gamma$ is virtually nilpotent if and only if the function $\gamma_{S}$ is bounded from above by a polynomial in $n$. Recently Shalom and Tao \cite{ShalomTao} obtained a strengthening of this theorem, namely that if $\gamma_{S}(n) \leq n^{c {(\log \log n)}^{c}}$ for some $n > 1/c$ with $c > 0$ a sufficiently small absolute constant, then $\Gamma$ is virtually nilpotent. 

The Gap Conjecture asserts that if a finitely generated group $\Gamma$ has growth type strictly less than $e^{\sqrt{n}}$ then it is virtually nilpotent (see \cite{Grigorchuk} for a precise formulation of the conjecture). As the famous Grigorchuk groups show this would be best possible even within the class of residually finite $p$-groups.

The above conjecture was shown to hold for residually nilpotent groups \cite{Grigorchuk}, \cite{LM}. For $\Gamma$ a residually solvable group the Gap Conjecture, with $\sqrt{n}$ replaced by $n^{1/7}$, has been proved by Wilson \cite{Wilson2} (see also \cite{Wilson1} and the slides \cite{WilsonTalk} of his talk at the 2010 Ischia Group Theory Conference). One of the main ingredients of the proof was to establish upper bounds for $\mathrm{diam}(G)$ in case $G$ is a solvable subgroup of $\mathrm{GL}(V)$ acting completely reducibly on the finite vector space $V$. Wilson shows that in general $\mathrm{diam}(G) \leq O(1)|V|$. He also points out that this bound is sharp since $G$ may be taken to be a Singer cycle in $\mathrm{GL}(V)$.   

Motivated by the above results we consider the diameters of arbitrary linear groups over finite vector spaces.

\begin{theorem}
\label{secondmain}
Let $G$ be a subgroup of $\mathrm{GL}(V)$ where $V$ is a vector space of dimension $n$ defined over the field of size $q$ and characteristic $p$. Let $h = \max_{S} \{ \mathrm{diam}(S) \}$ where $S$ runs over the (non-abelian) classical composition factors of $G$ defined over fields of characteristic $p$, if such exist, otherwise put $h=1$. Then $$\mathrm{diam}(G) < {|V|}^{O(1)}h^{2} < q^{O( n {(\log n)}^{2})}.$$   
\end{theorem}

Note that Theorem \ref{secondmain} may be viewed as an extension of Theorem \ref{firstmain}. Actually, both results also extend to directed Cayley graphs by a result of Babai \cite{B}.  

Theorem \ref{secondmain} is deduced from a structure theorem for a finite group acting completely reducibly on a vector space (see Theorem \ref{completelyreducible}).  

For $G = \mathrm{GL}(V)$ we must have $\mathrm{diam}(G) \geq \mathrm{diam}_{S}(G) \geq (q-1)/2$ where $S$ is a generating set of $G$ where all but one element in $S$ has determinant $1$. This shows that the diameter of absolutely irreducible (almost simple) subgroups of $\mathrm{GL}(n,q)$ may be much larger than the bound predicted by Babai's conjecture for $\mathrm{PSL}(n,q)$.

Kornhauser, Miller and Spirakis \cite{KMS} asked in 1984 whether or not the diameter of transitive groups is always polynomially bounded in terms of the degree. A positive answer (which is supported by the results in \cite{HS}) would show that the best possible bound for $\mathrm{S}_n$ and for its transitive subgroups is the same. (As the example of Singer cycles in $\mathrm{SL}(V)$ shows, the analogue of this is unlikely to be true for $\mathrm{SL}(V)$ where $V$ is a finite vector space.) Since the minimal degree of a permutation representation of a simple group of Lie type of rank $n$ over the field $\F_q$ is roughly $q^n$, our Theorem \ref{firstmain} also supports a positive answer to the above question.

\section{Proof of Theorem \ref{firstmain}}

\subsection{A new degree reduction lemma}
\label{2.1}

In this section we prove Theorem~\ref{firstmain}. To achieve this, we prove a new
degree reduction lemma for matrices over finite fields (Lemma~\ref{lem:deg_red}). 
This is a linear algebraic analogue of the degree reduction lemma for  
permutations by Babai and Seress \cite[Lemma 3]{BS87}. It improves the 
corresponding 
one by Biswas and 
Yang 
\cite[Lemma 4.4 (ii)]{BY}. Theorem~\ref{firstmain} then follows by combining 
Lemma~\ref{lem:deg_red} with the rest of the Biswas-Yang machinery. 

We first state our degree reduction lemma, and indicate how
Theorem~\ref{firstmain} follows from this together with \cite{BY}. We then prove 
this lemma in Section~\ref{subsec:proof_deg_red}. 

Let us set up some notation. Fix a finite field $\F_q$ of characteristic $p$. 
Let 
$\overline{\F_q}$ be the algebraic closure of $\F_q$. We 
use $I$ 
to denote identity matrices. Let $\M(n, q)$ denote the linear space 
of $n\times n$ matrices over $\F_q$, and $\GL(n, q)$ the group of $n\times n$ 
invertible matrices over $\F_q$. For $A\in \M(n, q)$, we use $\charp(A, x)$ 
and 
$\minp(A, x)$ to denote the characteristic polynomial and the minimal polynomial 
of 
$A$ in the variable $x$, respectively. The 
\emph{degree} of $A\in 
\M(n, q)$, denoted as $\deg(A)$, is 
defined 
to be the rank of $A-I$.

We now state the degree reduction lemma, whose proof is postponed to 
Section~\ref{subsec:proof_deg_red}.

\begin{lemma}\label{lem:deg_red}
Suppose we are given $A\in \GL(n, q)$, such that $\charp(A, x)$ has irreducible 
factors $f_{1}, \ldots , f_{r}$ of degrees $p_1, \dots, p_r$ respectively, where 
the $p_i$ are primes larger than $2$ for which the inequality 
$\prod_{i\in[r]}p_i>n^4$ holds. 
Then there exists $m\in \N$, such that $A^m$ 
is a non-identity matrix of degree at most $\deg(A)/4$. 
Furthermore, if each 
$f_{i}$ has a root of order $q^{p_{i}}-1$ 
over $\overline{\F_q}$, then there exists $m'\in \N$, such that $A^{mm'}$ has the 
additional property that $1$
is its only eigenvalue lying in $\F_q$. 
\end{lemma}

Note that an irreducible polynomial $f_{i}$ of degree $p_{i}$ over $\F_q$ has a root of order $q^{p_{i}}-1$ 
over $\overline{\F_q}$ if and only if $f_{i}$ is the minimal polynomial of some Singer cycle element in $\GL(p_{i},q)$. Such polynomials $f_{i}$ exist for every $p_{i}$ and $q$. 

Compare Lemma~\ref{lem:deg_red} with \cite[Lemma 4.4 (ii)]{BY}. The key 
difference is that Biswas and Yang 
required the primes to be coprime with $p(q-1)$, while we do not have such a 
restriction. This leads to the desired improvement, because of the following easy 
number-theoretic bounds, as already used in \cite[Sec. 3]{BS87}.

By a classical result of Erd\H os \cite{Erd32}, there exist constants $c_{1}$ and $c_{2}$ larger than $1$ such that for every number $x \geq 1$ we have $$c_{1}^{x} < \underset{p' \mathrm{prime}}{\underset{x < p' \leq 2x}{\prod}} p' < c_{2}^{x}.$$ For $y \geq 2$ let $f(y)$ be the product of all primes no greater than $y$. For $y \geq 4$ we have $c_{1}^{y/2} \cdot f(y/2) \leq f(y) \leq c_{2}^{y/2} \cdot f(y/2)$, and by induction this gives $c_{1}^{y} < f(y) < c_{2}^{y}$. Let $\bar{p}$ be a prime. From $f(\bar{p}) < {c_{2}}^{\bar{p}}$ we get  
$$\underset{p' \ \mathrm{prime}}{\underset{p' \leq \bar{p}}{\sum}} p' = \underset{p' \ \mathrm{prime}}{\underset{p' \leq \bar{p}}{\sum}} \Big( \frac{p'}{\log p'} \cdot \log p' \Big) < \frac{2 \bar{p}}{\log \bar{p}} \cdot \Big( 
\underset{p' \ \mathrm{prime}}{\underset{p' \leq \bar{p}}{\sum}} \log p' \Big) < \frac{2 {\bar{p}}^{2} \log c_{2}}{\log \bar{p}}.$$ 
For our purposes we may take $\bar{p}$ to be the smallest prime such that ${c_{1}}^{\bar{p}} \geq n^{4}$. This assures that the product of all primes no greater than $\bar{p}$ is larger than $n^{4}$ and also that the sum of all primes no greater than $\bar{p}$ is bounded by 
$$\frac{2 {\bar{p}}^{2} \log c_{2}}{\log \bar{p}} < c_{3} \frac{{(\log n)}^{2}}{\log \log n}$$ for some constant $c_{3}$ and all $n \geq 3$.
To see the latter claim note that 
$\bar{p} = O( \log n)$ by the Bertrand-Chebyshev theorem.

Compare the above estimates with \cite[Lemma 4.4 (i)]{BY}. There, 
because of the coprime with $p(q-1)$ condition, the sum over the orders of $q$ in 
$\Z/p_i\Z$ can only be bounded from above by $O((\log n+\log q)^3)$,
provided that 
the least common multiple of these orders is larger than $n^4$. 

\subsection{The Biswas-Yang machinery}

A proof of Theorem~\ref{firstmain} follows by plugging in Lemma~\ref{lem:deg_red} 
to the rest of the Biswas-Yang machinery 
\cite{BY}. We briefly outline the procedure.

In order to prove Theorem \ref{firstmain}, it is sufficient to assume that $G$ is a finite simple classical group (of unbounded dimension $n$), by the fact that Babai's conjecture is known to hold in the bounded rank case (see \cite{PSz} and \cite{BGT}). Moreover, it is sufficient to establish the estimate $\mathrm{diam}_{S}(G) < q^{O(n {(\log n)}^{2})}$ for every generating set $S$ of $G$ for every group $G$ isomorphic to $\SL(n, q)$, $\mathrm{Sp}(n, q)$, $\mathrm{SU}(n, q)$, or $\mathrm{\Omega}(n, q)$, with $n$ sufficiently large. 

Let $V$ be a vector space of dimension $n$ over the field $\F_{q}$. If $G$ is different from $\mathrm{SL}(V)$, we view $V$ as a non-degenerate formed space with a non-degenerate alternating bilinear form in the symplectic case, with a non-degenerate Hermitian form in the unitary case, or with a non-degenerate quadratic form in the orthogonal case.  

Let $t$ be a positive integer. Following \cite[Definition 2.1]{BY}, we say that a 
subset $H$ of $\mathrm{GL}(V)$ is a {\it $t$-transversal set} if, given any 
embedding $X$ of a subspace $W$ of dimension $t$ into $V$, there is a linear 
transformation in $H$ whose restriction to $W$ is $X$. If $V$ is equipped with a 
non-degenerate form, we say, following \cite[Definition 6.4]{BY}, 
that a subset $H$ of $G$ is a {\it singularly $t$-transversal set} if, for any 
isometric embedding $X$ of a totally singular subspace $W$ of dimension $t$ into 
$V$, there is an element of $H$ whose restriction to $W$ is $X$. Given any 
symmetric generating set $S$ for $G$, the set $S^{(t)} = \cup_{i=1}^{q^{nt}} 
S^{i}$ is $t$-transversal if $G = \mathrm{SL}(V)$ and $t < n$ (see \cite[Corollary 
2.4]{BY}) and is singularly $t$-transversal if $G \not= \mathrm{SL}(V)$ and $t 
\leq (n-2)/5$ (see \cite[Corollary 6.8]{BY}). 

The proof of Theorem \ref{firstmain} consists of two steps. The first step is Proposition \ref{classical}, which is \cite[Proposition 5.5]{BY} and \cite[Proposition 7.7]{BY} with different bounds.

\begin{proposition}
\label{classical}
There are universal positive constants $c_{4}$ and $c_{5}$ such that for any symmetric generating set $S$ in $G$ where $G$ is any of the groups 
$\SL(n, q)$, $\mathrm{Sp}(n, q)$, $\mathrm{SU}(n, q)$, $\mathrm{\Omega}(n, q)$, with $n > 2$, there is a non-scalar matrix $A$ in $G$ such that $\mathrm{deg}(A) <  c_{4} ({(\log n)}^{2}/\log \log n)$ and $A$ may be expressed as the product of less than $q^{c_{5} \cdot n \cdot {((\log n)}^{2}/\log \log n)}$ elements from $S$.   
\end{proposition}


\begin{proof}
We apply Lemma~\ref{lem:deg_red} to the argument of Biswas and Yang \cite{BY}. 

Let $G = \mathrm{SL}(V) = \mathrm{SL}(n,q)$. We may assume that $n$ is sufficiently large. Put $c_{4} = 2c_{3}$ and assume that $d$, defined to be the integer part of $c_{3} ({(\log n)}^{2}/\log \log n)$, is less than $n$. Since $S^{(d)}$ is a 
$d$-transversal set for $1 \leq d<n$, there is 
$A_0\in S^{(d)}$ 
that 
maps some $d$-dimensional subspace $W$ to itself, and the restriction of 
$A_0$ to $W$ is a 
diagonal block matrix $C$, where the blocks are companion matrices of irreducible 
polynomials $f_{i}$ of degrees $p_i$, and possibly an identity matrix of an appropriate 
size, such that the $p_i$ range over all primes from $3$ to $\bar p$ as in
Section \ref{2.1} and each $f_{i}$ has a root of order $q^{p_{i}}-1$ over $\overline{\F_q}$. 
Then $A_0$ satisfies the condition of Lemma~\ref{lem:deg_red}, and the 
length of $A_0$ is bounded by 
$q^{nd}$. By Lemma~\ref{lem:deg_red}, raise $A_0$ to an appropriate power to 
obtain a non-identity 
matrix $A_1$ of degree at most $\deg(A_0)/4$ with eigenvalues being either $1$ or 
outside $\F_q$. The length of $A_1$ is bounded by $q^{nd + n}$ since the 
order 
of $A_0$ is bounded by $q^n$. If $\deg(A_1)< 2d$, then we are done. Otherwise, 
we 
enter 
the inductive step. The key in the inductive step is to locate a subspace $W_1$ of 
dimension $d$ such that $A_1W_1\cap W_1=0$, whose existence is guaranteed by 
\cite[Lemma 5.3]{BY}. Then 
use the $2d$-transversal set $S^{(2d)}$ to 
obtain a matrix $M_1$ of length at most $q^{2nd}$ that fixes $A_1W_1$ pointwise, $W_{1}$ setwise, and when 
restricting to $W_1$, realises 
the diagonal block $C$ as before. The commutator $A_1'=M_1A_1^{-1}M_1^{-1}A_1$ 
then realises $C$ when restricting on $W_1$, so it satisfies the condition of 
Lemma~\ref{lem:deg_red}. Furthermore, $\deg(A_1')\leq 2\deg(A_1)$ by 
\cite[Proposition 5.2]{BY}. By 
Lemma~\ref{lem:deg_red}, raise $A_1'$ to an appropriate power to get a 
non-identity matrix $A_2$ such that $$\deg(A_2)\leq \deg(A_1')/4\leq 
2\deg(A_1)/4=\deg(A_1)/2.$$ It can be checked that the length of $A_2$ is 
bounded by $$2 (q^{2nd} + q^{nd + n}) q^{n} \leq q^{2nd + 2(n+2)}.$$
Suppose we have obtained a non-scalar matrix $A_{j}$ with eigenvalues either $1$ or outside $\F_{q}$ with $\deg(A_{j}) \leq n/2^{j+1}$ and length at most $q^{2nd + j(n+2)}$. If $\deg(A_{j})$ is not small enough, then we construct a matrix $A_{j+1}$ of length at most 
$$2 (q^{2nd} + q^{2nd + j(n+2)}) q^{n} \leq q^{2nd + (j+1)(n+2)}.$$ Repeat this by at most $\log n$ times to reach the desired matrix $A$.

For $G \not= \mathrm{SL}(V)$ the argument is very similar as for $\mathrm{SL}(V)$ 
above. Here Witt's decomposition theorem (see \cite[Theorem 6.2]{BY}) and Witt's 
extension lemma (see \cite[Lemma 6.5]{BY}) are used. The latter is that 
$G$ is a singularly $t$-transversal set for any $t$. Moreover, we mention 
\cite[Lemma 7.6]{BY}. If $A$ is a matrix in $G$ of degree $d$ such that the 
eigenvalues of $A$ are either $1$ or outside $\F_q$, then there is a totally 
singular subspace $W$ of $V$ such that $W \cap AW = \{ 0 \}$, $W \perp AW$, and 
$\dim W \geq (d/32) - (7/4)$.  
\end{proof}

Given a non-scalar matrix $A$ of degree $d$ and length $\ell$, the 
second step is to show that the diameter of $G$ with respect to $S$ is bounded by 
$O((q^{2nd}+\ell) \cdot \frac{n}{d})$ (cf. \cite[Proposition 8.3]{BY}). This is 
due to the 
following. Firstly, any conjugate of 
$A$ can be obtained by 
conjugating by a matrix of length less than $q^{2nd}$, as the number of 
conjugates of $A$ is bounded by such (see \cite[Lemma 8.1]{BY}). Secondly, by Liebeck and Shalev 
\cite{LS}, every element in $G$ is a product of at most $O(n/d)$ 
conjugates of $A$. 

We may take $d$ to be less than $c_{4} ({(\log n)}^{2}/\log \log n)$ and $\ell$ to be less than $q^{c_{5} \cdot n \cdot {((\log n)}^{2}/\log \log n)}$ by Proposition \ref{classical}. Then $$\mathrm{diam}_{S}(G) \leq O ((q^{2nd}+\ell) \cdot \frac{n}{d}) \leq q^{O(n {(\log n)}^{2})}.$$ This completes the proof of Theorem \ref{firstmain} (modulo Lemma \ref{lem:deg_red}).

We remind the reader that in the above procedure, the exponent with respect to the 
base $q$ 
in the 
length bound of $A$ is always bounded by $O(nd)$. It follows that the $\log q$ 
term does not 
appear in the 
exponent if $d=O({(\log n)}^{2})$.

\subsection{Proof of Lemma~\ref{lem:deg_red}}\label{subsec:proof_deg_red}


We first need the following preparations. 

\begin{fact}\label{fact:comp}
Let $f=f(x)\in \F_q[x]$ be an irreducible monic polynomial of degree $d$. Let 
$C_f\in 
\GL(d, q)$ be its companion matrix. 
\begin{enumerate}
\item For any $a\in \N$, $C_f^{p^a}$ is similar to the companion matrix of an 
irreducible 
polynomial in $\F_q[x]$ of degree $d$.
\item For $m\in \N$, $C_f^{q^m-1}=I$ if and only if $d\mid m$. 
\end{enumerate}
\end{fact}
\begin{proof}
(1) First observe that $f^{p^a}(x)=\tilde{f}(x^{p^a})$ where $\tilde{f}$ is the 
polynomial obtained by raising every coefficient of $f$ to the $p^a$th power. Then 
we can verify that $\charp(C_f^{p^a}, x)=\tilde{f}(x)$.

(2) Recall that $\minp(C_f, x)=\charp(C_f, x)=f(x)$, and $f(x)\mid x^{q^m}-x$ if 
and only if $d\mid m$. The claim then follows. 
\end{proof}

\begin{theorem}[{\cite{Mal63}, generalized Jordan normal form}]\label{thm:maltsev}
Let $\F$ be a perfect field, and $A\in M(n,\F)$. Suppose $\charp(A, x)$ decomposes 
into a product of irreducible monic polynomials as $f_1^{e_1}\cdot\ldots\cdot 
f_k^{e_k}$, 
where $f_i\in\F[x]$ is of degree $d_i$. Then $A$ is similar to a 
block diagonal matrix $\diag(J_1, \dots, J_\ell)$, 
where each $J_i$, called a (generalized) Jordan block, is of the form 
\begin{equation}\label{eqn:jblock}
\begin{bmatrix}
C_{f_{b_i}} & I & 0  & \dots & 0 & 0  \\
0 & C_{f_{b_i}} & I  & \dots & 0 & 0 \\
\vdots & \vdots  & \vdots & \ddots & \vdots & \vdots \\
0 & 0 & 0 & \dots & I & 0 \\
0 & 0 & 0 & \dots & C_{f_{b_i}} & I \\
0 & 0 & 0 & \dots & 0 & C_{f_{b_i}}
\end{bmatrix},
\end{equation}
where $b_i\in[k]$, $I$ is the identity matrix of size $d_{b_i}$, and $0$ is the 
all-zero matrix of size $d_{b_i}\times d_{b_i}$. 
\end{theorem}



We are ready to prove Lemma~\ref{lem:deg_red}. 


\begin{proof}[Proof of Lemma~\ref{lem:deg_red}]
Suppose $\charp(A, x)$ decomposes 
into a product of irreducible monic polynomials as 
$$f_1^{e_1}\cdot\ldots\cdot 
f_k^{e_k}\cdot (x-1)^t,$$ 
where $t\in \N$, $f_i\in\F[x]$ is irreducible, monic, and of degree $d_i$, 
and $f_i\neq x-1$ for $i\in[k]$. 
Let $f_0=x-1$, and $s=n-t$. 
Clearly, 
$\deg(A)\geq s$. 
By our assumption, we can assume that $k\geq r$ and $\deg(f_i)=p_i$ for $i\in[r]$. 

For our purpose, we can replace $A$ with any of its conjugates. Therefore by 
Theorem~\ref{thm:maltsev}, we assume $A=\diag(J_1, \dots, J_\ell)$ 
where each $J_i$ 
is a Jordan block of the form (\ref{eqn:jblock}).

We first raise $A$ to the $p^a$th power, where $a$ is an integer larger than 
$\log n$. Then for any $i\in[\ell]$, $J_i^{p^a}\cong \diag(C_{\tilde{f}_{b_i}}, 
\dots, 
C_{\tilde{f}_{b_i}})$ 
for some $b_i\in\{0, 1, \dots, k\}$, where 
$\tilde{f}_{b_i}$ 
is an irreducible polynomial of 
degree $d_{b_i}$ by Fact~\ref{fact:comp} (1). Let 
$\tilde{A}=A^{p^a}$. By 
arranging the 
diagonal blocks via conjugation transformations, we can assume that 
$$\tilde{A}=\diag(C_{\tilde{f}_1}, \dots, 
C_{\tilde{f}_1}, \dots, C_{\tilde{f}_k}, \dots, C_{\tilde{f}_k}, 1, \dots, 1),$$
where the number of $C_{\tilde{f}_i}$ is $e_i$, and the number of $1$ is $t$. 
In particular, $\deg(\tilde A)=s\leq \deg(A)$. 

For $j\in [s]$, 
let $c_j\in [k]$ be such that the $j$th diagonal entry (not blocks) of $\tilde{A}$ 
is in the 
diagonal block 
$C_{\tilde{f}_{c_j}}$. We then build a zero-one matrix $D$ 
of size $r\times s$ as follows. For $i\in[r]$ and $j\in[s]$, $D(i, j)=1$ if 
$p_i|d_{c_j}$, and $0$ otherwise. We then deduce the following. 
\begin{enumerate}
\item[(a)] For any $j\in[s]$, $\prod_{i\in[r]}p_i^{D(i,j)}\leq d_{c_j}\leq s$. 
\item[(b)] 
For $i\in[r]$, let 
$n_i=\sum_{j\in[s]}D(i,j)$. We claim that there exists $i'\in[r]$, such that 
$n_{i'}\leq s/4$. For this, consider the weighted average $W$ of $n_i$ with 
weights $\log p_i$. We have 
\begin{eqnarray*}
W & = & \frac{\sum_{i\in[r]} n_i\log p_i} {\sum_{i\in[r]}\log p_i} = 
\frac{\sum_{j\in[s]} \sum_{i\in[r]}D(i,j)\log p_i}{\sum_{i\in[r]}\log p_i} 
 \\
 & \leq & \frac{\sum_{j\in[s]} \sum_{i\in[r]} D(i,j)\log p_i}{4\log n} \\
 & = & \frac{\sum_{j\in[s]} \log (\prod_{i\in[r]}p_i^{D(i,j)})}{4\log n} \\
 & \leq & \frac{s\cdot \log s}{4\log n} \leq \frac{s}{4}.
\end{eqnarray*}
In the above, the first $\leq$ is due to the choice of the $p_i$, namely we have 
chosen 
those $p_i$ to satisfy $\prod_{i\in[r]}p_i>n^4$. The second $\leq$ is due to 
item 
(a) we just described. The existence of such $i'\in[r]$ satisfying $n_{i'}\leq 
s/4$ 
then follows. 
\end{enumerate}

Let $i' \in [r]$ be an index satisfying (b) and let $p' = p_{i'}$. Let $s'$ be the lowest common multiple of these $\mathrm{deg}(f_{i})$ which are coprime to $p'$. Then
$\hat{A}=\tilde{A}^{q^{s'}-1}$ satisfies the following.
Firstly, 
$\hat{A}$ is not identity. This is because the existence of $C_{\tilde f_{i'}}$ 
where $\deg(\tilde f_{i'})=p_{i'}$ and Fact~\ref{fact:comp} (2). Secondly, 
$\hat{A}$ 
is of degree at most $s/4\leq \deg(A)/4$. 
This is because for 
any $C_{\tilde f_i}$ with 
$p_{i'}\nmid 
\deg(\tilde f_i)$, $C_{\tilde f_i}^{q^{s'}-1}$ becomes identity by 
Fact~\ref{fact:comp} (2), and from (b) we know the sum of the 
sizes of such blocks is at least $(3s)/4$. This shows the existence of $m\in \N$ 
as required. 

We now prove the statement of the furthermore part 
in Lemma \ref{lem:deg_red}. For 
this, it is sufficient to show that a ${(q-1)}^{k}$-power of $\hat{A}$, for some 
integer $k$, is not the identity matrix. Suppose otherwise. If the 
${(q-1)}^{k}$-power of $\hat{A}$ is the identity, then, in particular, the 
$(q^{s'}-1){(q-1)}^{k}$-power of the companion matrix $C_{\tilde f_i}$ is the 
identity. Since $C_{\tilde f_i}$ has order $q^{p'}-1$ by the assumption 
on the root orders of $f_i$ (taking the 
$p^{a}$-power of $A$ does not do harm), we must have $q^{p'}-1 \mid 
(q^{s'}-1){(q-1)}^{k}$. Since the greatest common divisor of $q^{p'}-1$ and 
$q^{s'}-1$ is $q-1$, it follows that $q^{p'}-1 \mid {(q-1)}^{k+1}$. By applying 
Claim \ref{that} with $p'$ in place of $t$ and noting that $p' > 2$, we arrive to 
a contradiction. 

\begin{claim}
\label{that}
Let $t$ be a prime and $q$ an integer larger than $1$. If $q^{t}-1$ divides some power of $q-1$, then $t=2$.
\end{claim}

\begin{proof}
Notice that the condition $q^{t}-1$ divides some power of $q-1$ is equivalent to the condition that every prime divisor of $q^{t}-1$ divides $q-1$. 

We claim that $\frac{q^{t}-1}{q-1} = t^{s}$ for some integer $s \geq 2$. Let $r$ be a prime divisor of $(q^{t}-1)/(q-1)$. Then $r$ divides $q-1$ by our condition and, since $q^{t-1} + \cdots + q +1$ is congruent to $t$ modulo $q-1$, the primes $r$ and $t$ must be equal. This proves that $\frac{q^{t}-1}{q-1} = t^{s}$ for some integer $s \geq 1$. We also have $s \geq 2$ by $q > 1$. 

On the other hand, $$\frac{q^{t}-1}{q-1} = \frac{{((q-1)+1)}^{t}-1}{q-1} = 
\sum_{k=1}^{t} \binom{t}{k} {(q-1)}^{k-1}$$ is congruent to ${(q-1)}^{t-1} + t$ 
modulo $t^{2}$, 
as the intermediate terms 
$\binom{t}{k}(q-1)^{k-1}$, $1<k<t$, are divisible by $t^2$ if they do appear. 
Since $t^{2}$ does not 
divide $t$, it cannot divide ${(q-1)}^{t-1}$ either. This forces $t = 2$ as $t$ 
divides $q-1$ by our condition.   
\end{proof}
This concludes the proof of Lemma~\ref{lem:deg_red}.
\end{proof}


\section{A structure theorem for completely reducible groups}

In this section we will prove Theorem \ref{completelyreducible} which, in the next section, will be used to deduce Theorem \ref{secondmain} (in case the group acts completely reducibly on its module). 

Let us fix some notation. Let $V$ be the vector space of dimension $n$ over $\F_q$. Let $G$ be a subgroup of $\mathrm{GL}(V)$ acting completely reducibly on $V$. The $G$-module $V$ is the direct sum $V_{1} \oplus \cdots \oplus V_{m}$ of irreducible $G$-modules $V_{i}$ with $1 \leq i \leq m$. It is natural to write each vector space $V_{i}$ as a direct sum $W_{i1} \oplus \cdots \oplus W_{ik_{i}}$ of isomorphic vector 
spaces $W_{ij}$ with $1 \leq j \leq k_{i}$ such that $\{ W_{i1}, \ldots , W_{ik_{i}} \}$ is preserved by the action of $G$ and with $k_{i}$ as large as possible. It follows that for each pair $(i,j)$ the stabilizer of $W_{ij}$ in $G$ acts irreducibly and primitively (but not necessarily faithfully) on $W_{ij}$. 

To simplify notation, write the vector space $V$ as a direct sum $W_{1} \oplus \cdots \oplus W_{k}$ such that $G$ preserves $\Omega = \{ W_{1}, \ldots , W_{k} \}$, the stabilizer $G_{i}$ of $W_{i}$ in $G$ acts irreducibly and primitively on $W_{i}$ for each $i$ with $1 \leq i \leq k$ and $k = \sum_{i=1}^{m} k_{i}$ in the above notation. For each $i$ let the action of $G_{i}$ on $W_{i}$ be $P_{i}$. The group $G$ is a subgroup of $(P_{1} \times \cdots \times P_{k}):\mathrm{S}_{k}$. Let $N$ denote the intersection of $G$ with $P_{1} \times \cdots \times P_{k}$, that is, the kernel of the action of $G$ on $\Omega$. The factor group $G/N$ may be viewed as a subgroup of $\mathrm{S}_{k} \leq \mathrm{S}_{n}$. 

We continue with a slightly simplified version of \cite[Proposition 5.7]{JP}. Here a {\it quasisimple} group is a finite perfect group $H$ such that $H/Z(H)$ is simple.   

\begin{theorem}[Jaikin-Zapirain, Pyber; 2011]
\label{theoremJP0}
Let $Q$ be a subgroup of $\mathrm{GL}(W)$ with $Q$ acting irreducibly and primitively on the finite vector space $W$ defined over the prime field $\F_p$. For the generalized Fitting subgroup $F^{*}(Q)$ of $Q$ let $F$ be the field $Z(\mathrm{End}_{F^{*}(Q)}(W))$. There exists a universal constant $c_{6}$ such that whenever $|Q| > {|W|}^{c_{6}}$, then
\begin{enumerate}
\item[(i)] there is a tensor product decomposition $U' \otimes_{F} U$ of $W$ such that $\dim(U) \geq \dim(U')$;

\item[(ii)] there is a quasisimple normal subgroup $R$ in $Q$ isomorphic to $\mathrm{A}_{\ell}$ or to a classical group $\mathrm{Cl}(d,K)$ for some $K \leq F$;

\item[(iii)] if $R = \mathrm{A}_{\ell}$, then $U$ is the natural $\mathrm{A}_{\ell}$-module, while if $R = \mathrm{Cl}(d,K)$, then $U$ is $F \otimes_{K} U''$ where $U''$ is the natural $\mathrm{Cl}(d,K)$-module;

\item[(iv)] $|Q/R| \leq {|W|}^{5}$.
\end{enumerate}
\end{theorem}

Let $P$ be a subgroup of $\mathrm{GL}(W)$ acting irreducibly and primitively on the finite vector space $W$ defined over the field $\F_q$ (possibly different from its prime field $\F_p$). It centralizes a cyclic subgroup $Z$ of $\mathrm{GL}(W)$ isomorphic to $\F_{q}^{*}$. According to a claim of Liebeck and Shalev (see \cite[p. 112]{LiebeckShalev}) $PZ$ acts irreducibly and primitively on $W$ viewed over the field $\F_p$. For the sake of completeness, we present a proof for this fact. If $U$ is a $PZ$-invariant subspace of $W$, then $U$ must be an $\F_q$-space. Thus $PZ$ acts irreducibly on $W$. Let $W = W_{1} + \cdots + W_{t}$ be an imprimitivity decomposition of the $PZ$-module $W$ over $\F_p$ where $t > 1$. Let $Z_{0}$ be the stabilizer of $W_1$ in $Z$. Clearly $Z_{0} < Z$ since otherwise the $W_i$ are $\F_q$-spaces contradicting the fact that $P$ acts primitively on $W$ viewed over $\F_q$. Let $z$ be an element of $Z$ mapping $W_1$ to $W_2$ and let $w_{1}$ be a non-zero vector in $W_1$. Consider the element $1+z$ inside $\F_q$. Since $z \not= -1$, $1+z \in Z$ and $w_{1}(1+z) = w_{1} + w_{1}z \in W_{1} + W_{2}$. Since $w_{1} \not= 0$, the element $w_{1}(1+z)$ is neither in $W_{1}$ nor in $W_{2}$. This is a contradiction.  

As a Corollary to Theorem \ref{theoremJP0} we obtain the surprising fact that primitive linear groups are not far from being simple groups.

\begin{theorem}
\label{theoremJP}
If $P$ is a subgroup of $\mathrm{GL}(W)$ with $P$ acting irreducibly and primitively on the finite vector space $W$ defined over the field $\F_q$ with $|P| > {|W|}^{c_{6}}$, then there is a quasisimple normal subgroup $R$ in $P$ isomorphic to $\mathrm{A}_{\ell}$ such that $\ell \leq \dim_{\F_q}(W)$ or to a classical group $\mathrm{Cl}(d,r)$ such that $d \leq \dim_{\F_q}(W)$ with $\F_{r}$ and $\F_{q}$ of the same characteristic, and $|P/R| \leq {|W|}^{5}$. Moreover, if $R$ is isomorphic to $\mathrm{Cl}(d,r)$, then $r^{d} \leq |W|$.
\end{theorem}

\begin{proof}
By Theorem \ref{theoremJP0} and the claim of Liebeck and Shalev (see the paragraph after Theorem \ref{theoremJP0}), there is a quasisimple normal subgroup $R$ in $PZ$ isomorphic to $\mathrm{A}_{\ell}$ such that $\ell \leq \dim_{\F_q}(W)$ or to a classical group $\mathrm{Cl}(d,r)$ such that $d \leq \dim_{\F_q}(W)$ (the bounds for $\ell$ and $d$ follow from the fact that the field $F$ in Theorem \ref{theoremJP0} contains $\F_{q}$) and $\F_{r}$ and $\F_{q}$ have the same characteristic. In the latter case we have $r^{d} \leq |W|$ by Theorem \ref{theoremJP0}. It also follows that $|PZ/R| \leq {|W|}^{5}$. Since $R$ is quasisimple, $R = [R,R] \leq [PZ,PZ] \leq P$. This completes the proof of the theorem.  
\end{proof}

We are now in position to prove our structure theorem.  

\begin{theorem}
\label{completelyreducible}
Let $V$ be a vector space of dimension $n$ over the field $\F_q$. Let $G \leq \mathrm{GL}(V)$ be a group acting completely reducibly on $V$. Write $V$ as a direct sum $W_{1} \oplus \cdots \oplus W_{k}$ of (non-trivial) subspaces of $V$ in such a way that $G$ permutes the set $\Omega = \{ W_{1}, \ldots , W_{k} \}$ and the stabilizer of each $W_i$ in $G$ acts primitively on $W_{i}$ for every $i$ with $1 \leq i \leq k$. Let $N$ be the kernel of the action of $G$ on $\Omega$. In particular, $G/N$ may be viewed as a subgroup of $\mathrm{S}_n$. There exists a constant $c_{7}$ such that whenever $|N| > {|V|}^{c_{7}}$, 
\begin{enumerate}
\item[(i)] there is a normal subgroup $C$ of $G$ contained in $N$ such that $C = Q_{1} \circ \cdots \circ Q_{w}$ is a central product of quasisimple groups $Q_{i}$ with $w \leq k$;

\item[(ii)] each $Q_i$ has a factor group $T_i$ such that for some $j\in\{1, 
\dots, k\}$, $T_{i}$ is an 
alternating group $\mathrm{A}_{\ell_{j}}$ with $\ell_j\leq 
\dim_{\F_q}(W_{j})$, or $T_{i}$ is a classical simple group $\mathrm{Cl}(d_{j},r_{j})$ such that 
$\F_{r_{j}}$ and $\F_{q}$ have the same characteristic, 
$d_{j} \leq \dim_{\F_q}(W_{j})$ and $r_{j}^{d_{j}} \leq |W_{j}|$;


\item[(iii)] $|N/C| \leq {|V|}^{c_{7}}$.
\end{enumerate}
\end{theorem}

\begin{proof}
Let $c_{7}$ be the maximum of $6$ and $c_{6}$. Without loss of generality, we may assume that there is an integer $t \geq 0$ such that $|P_{i}| > {|W_{i}|}^{c_{7}}$ for every $i$ with $i \leq t$ and $|P_{i}| \leq {|W_{i}|}^{c_{7}}$ for every $i$ with $t < i \leq k$. For every $i$ with $i \leq t$, let $R_{i}$ be the quasisimple normal subgroup of $P_{i}$ whose existence is assured by Theorem \ref{theoremJP} (and is $R$ in that notation). 

If $N$ denotes the intersection of $G$ with $P_{1} \times \cdots \times P_{k}$, that is, the kernel of the action of $G$ on $\Omega$, then the factor group $G/N$ may be viewed as a subgroup of $\mathrm{S}_{k} \leq \mathrm{S}_{n}$. In order to prove the theorem, we may assume that $|N| > {|V|}^{c_{7}}$. In particular, $t>0$. 

Let $M$ be the normal subgroup of $G$ defined to be the intersection of $N$ and $R_{1} \times \cdots \times R_{t}$. Since the natural projection $M_{i}$ of $M$ to $P_{i}$ is normal in $P_{i}$, the group $M_{i}$ must also be normal in $R_{i}$. Since $R_{i}$ is quasisimple, $M_{i} = R_{i}$ or $M_{i}$ is central in $R_{i}$. In the latter case $|M_{i}| < |W_{i}|$. Without loss of generality, we may assume that there is a $u \geq 0$ such that $M_{i} = R_{i}$ for every index $i$ at most $u$ and $M_{i}$ is abelian for $i > u$. Thus the commutator subgroup $M'$ may be viewed as a subgroup of $R_{1} \times \cdots \times R_{u}$ where $u \geq 0$ which projects onto $R_{i}$ for every $i$ with $i \leq u$. Clearly, $|N/M'| \leq {|V|}^{c_{7}}$ by Theorem \ref{theoremJP}. 

We may thus assume that $M' \not= 1$, that is, $u \geq 1$. Now $M'/Z(M')$ may be viewed as a subgroup of $F_{1} \times \cdots \times F_{u}$ where $F_{i} = R_{i}/Z(R_{i})$ is a non-abelian simple group for every $i$ with $1 \leq i \leq u$. Moreover, $M'/Z(M')$ projects onto every $F_i$. It follows, by \cite[p. 328, Lemma]{S}, that $M'/Z(M')$
is a direct product $\prod_{j=1}^{w} D_j$ of full diagonal subgroups $D_{j}$ of subproducts $\prod_{i \in I_{j}} F_{i}$ where the $I_{j}$ form a partition of $\{1, \ldots , u \}$. The preimage in $M'$ of any simple factor $D_{j}$ of $M'/Z(M')$ contains a normal quasisimple subgroup of $M'$ which is subnormal in $G$. Let $C$ be the product of all components, that is, all subnormal quasisimple subgroups, of $G$ contained in the group $M'$. Since any two distinct components in a finite group commute, $C$ may be expressed in the form $Q_{1} \circ \cdots \circ Q_{w}$ where the $Q_{j}$ are components of $G$ contained in $M'$. 

The group $C$ is normal in $G$ and so (i) is established. Since $C \cdot Z(M') = M'$, it is easy to see that there is a refinement of our previous bound for $|N/M'|$ in the form $|N/C| \leq {|V|}^{c_{7}}$. This is (iii). 

Fix an index $i$ at most $w$. The component $Q_{i}$ projects onto $F_{j}$ for some $j$ at most $u$. The group $F_{j}$ is isomorphic to $\mathrm{A}_{\ell_{j}}$ such that $\ell_{j} \leq \dim_{\F_q}(W_{j})$ or to a classical simple group $\mathrm{Cl}(d_{j},r_{j})$ such that $d_{j} \leq \dim_{\F_q}(W_{j})$, $r_{j}^{d_{j}} \leq |W_{j}|$, and $\F_{r_{j}}$ and $\F_{q}$ have the same characteristic, by Theorem \ref{theoremJP}. Thus $Q_{i}$ has a factor group $T_{i}$ such that $T_{i}$ is $\mathrm{A}_{\ell_{j}}$ or $T_{i}$ is the classical simple group $\mathrm{Cl}(d_{j},r_{j})$. This gives (ii).



This completes the proof of the theorem.
\end{proof}

\section{A bound for $\mathrm{diam}(G)$ for $G$ a linear group}

In this section we prove Theorem \ref{secondmain}. 

A main tool in our argument is Lemma 5.1 of Babai and Seress \cite{BS}.

\begin{lemma}[Babai, Seress; 1992]
\label{BabaiSeressLemma}
If $N$ is a non-trivial, proper normal subgroup in a finite group $G$, then $\mathrm{diam}(G) \leq 4 \cdot \mathrm{diam}(N) \cdot \mathrm{diam}(G/N)$.
\end{lemma}

Now let $G$ be a subgroup of $\mathrm{GL}(V)$ acting on the finite vector space $V$ of dimension $n$ over the field of size $q$ and characteristic $p$. In case $h \not= 1$ let $S$ be a classical (non-abelian) composition factor of $G$ defined over a field of characteristic $p$ such that $h = \mathrm{diam}(S)$. 

First assume that $G$ acts completely reducibly on $V$. In this case we rely on Theorem \ref{completelyreducible} to prove Theorem \ref{secondmain}.

We use the notation of Theorem \ref{completelyreducible}. Theorem 1.3 of Babai and Seress \cite{BS} implies that $\mathrm{diam}(G/N)$ is less than exponential in $n$. Thus, in order to establish our bound for $\mathrm{diam}(G)$, it is sufficient to show that $\mathrm{diam}(N) < {|V|}^{O(1)}h^{2} < q^{O( n {(\log n)}^{2})}$ by Lemma \ref{BabaiSeressLemma}. This is certainly true in case $|N| \leq {|V|}^{c_{7}}$. Thus assume that $|N| > {|V|}^{c_{7}}$. Let $C$ be the normal subgroup of $G$, as in Theorem \ref{completelyreducible}, such that $|N/C| < {|V|}^{c_{7}}$. It follows by Lemma \ref{BabaiSeressLemma} that it is sufficient to show that $$\mathrm{diam}(C) < {|V|}^{O(1)}h^{2} < q^{O( n {(\log n)}^{2})}.$$ This paragraph also shows that $h < {|V|}^{O(1)}$ or $S$ is a composition factor of $C$. 

Since $C$ is normal in $G$, the center $Z(C) \leq \mathrm{GL}(V)$ of $C$ is an abelian group acting completely reducibly on $V$. By Schur's lemma and the fact that a finite division ring is a field, an abelian group $A \leq \mathrm{GL}(W)$ acting irreducibly on a finite vector space $W$ is cyclic and has order at most $|W|-1$. From this it follows that $|Z(C)| < |V|$. The factor group $C/Z(C)$ is the direct product of non-abelian simple groups each isomorphic to an alternating group or to a classical group in characteristic $p$. Let $A$ be the product of all factors of $C/Z(C)$ which are isomorphic to alternating groups, if such exist, otherwise let $A = 1$. Let $B$ be the product of all other simple factors of $C/Z(C)$, that is, $C/Z(C) = A \times B$. Notice that it is sufficient to establish the bound $\mathrm{diam}(A \times B) < {|V|}^{O(1)}h^{2} < q^{O( n {(\log n)}^{2})}$. 

The sum of degrees of all simple factors in $A$, if such exist, is at most $n$ by Theorem \ref{completelyreducible}. Hence $A$ may be considered as a permutation group of degree at most $n$ and hence $\mathrm{diam}(A) < O(1) |V|$ by Theorem 1.3 of \cite{BS}. It is then sufficient to see that $\mathrm{diam}(B) < {|V|}^{O(1)}h^{2} < q^{O( n {(\log n)}^{2})}$, by Lemma \ref{BabaiSeressLemma}. 

We have $\mathrm{diam}(B) \leq 20 \ n^{3} \ h^2 < {|V|}^{O(1)}h^{2}$ by \cite[Lemma 5.4]{BS}. Thus it is sufficient to establish $h = q^{O( n {(\log n)}^{2})}$. We may assume by the above that $S$ is a composition factor of $C$ (and a direct factor of $B$). In this case $S$ is isomorphic to the non-abelian composition factor $S_{i}$ of some component $Q_{i}$ of $G$ (normal in $C$). The group $S_{i}$ is a simple classical group of dimension $d_{j}$ defined over the field $\F_{r_{j}}$, for some $j$. Thus $h = r_{j}^{O( d_{j} {(\log d_{j})}^{2} )}$ by Theorem \ref{firstmain}. Since $d_{j} \leq n$ and $r_{j}^{d_{j}} \leq q^{n}$, we conclude that $r_{j}^{O( d_{j} {(\log d_{j})}^{2} )} =  q^{O( n {(\log n)}^{2} )}$. 

This completes the proof of Theorem \ref{secondmain} when $G$ acts completely reducibly.

Now let $G$ be an arbitrary subgroup of $\mathrm{GL}(V)$. Let $O_{p}(G)$ denote the largest normal $p$-subgroup of $G$. The factor group $G/O_{p}(G)$ may be viewed as a completely reducible linear group acting on the direct sum of the composition factors of the $G$-module $V$. Thus $\mathrm{diam}(G/O_{p}(G)) < {|V|}^{O(1)} h^{2}  < q^{O( n {(\log n)}^{2})}$ by the above.

In order to complete the proof of Theorem \ref{secondmain}, it is sufficient, by Lemma \ref{BabaiSeressLemma}, to show that $\mathrm{diam}(P) < {|V|}^{O(1)}$ for every $p$-subgroup $P$ of $\mathrm{GL}(V)$. 

Let $Q$ be a $p$-group and $\mathcal{C}$ a normal chain in $Q$ such that every associated factor in the chain $\mathcal{C}$ is elementary abelian. Let $\ell(Q, \mathcal{C})$ be the length of the chain $\mathcal{C}$ and let $r(Q, \mathcal{C})$ be the maximum rank of the associated factors in $\mathcal{C}$. It is easy to see that 
\begin{equation}
\label{utolso}
\mathrm{diam}(Q) \leq 4^{\ell(Q, \mathcal{C}) - 1} \cdot {(p \cdot r(Q, \mathcal{C}))}^{\ell(Q, \mathcal{C})}
\end{equation}
using Lemma \ref{BabaiSeressLemma} and Lemma 5.2 of \cite{BS}.  

Let $m$ be the smallest power of $2$ which is larger than $n$. An arbitrary subgroup $P$ of $\mathrm{GL}(V)$ may be viewed as a subgroup of a Sylow $p$-subgroup $S$ of $\mathrm{GL}(m,q)$. We have $$\mathrm{diam}(P) \leq 4^{\ell(S, \mathcal{C}) - 1} \cdot {(p \cdot r(S, \mathcal{C}))}^{\ell(S, \mathcal{C})}$$ by (\ref{utolso}), for any chain $\mathcal{C}$ of normal subgroups in $S$ such that the associated factor groups are elementary abelian. There exists an elementary abelian normal subgroup $A$ in $S$ such that $|A| = q^{m^{2}/4}$ and $S/A$ is the direct product of two copies of a Sylow $p$-subgroup in $\mathrm{GL}(m/2,q)$. It follows, by induction on $m$, that there is a chain $\mathcal{C}$ of normal subgroups in $S$ such that (i) the associated factor groups are elementary abelian; (ii) the first group is $A$; (iii) $r(S, \mathcal{C}) = (m^{2}/4) \cdot \log_{p}q \leq n^{2} \cdot \log_{p}q$; and (iv) $\ell(S, \mathcal{C}) = 1+ \log_{2}m \leq 2 + \log_{2}n$. From this it follows that $\mathrm{diam}(P) < {|V|}^{O(1)}$.

This completes the proof of Theorem \ref{secondmain}.


\end{document}